\documentclass[12pt,a4paper, notitlepage]{article}
\usepackage[a4paper] {geometry}
\usepackage{amsmath,amsfonts,amssymb,mathrsfs,amsthm}
\usepackage{multirow}
\usepackage[english]{babel}
\usepackage[all]{xy}
\usepackage{mathtools}
\numberwithin{equation}{section}
\numberwithin{figure}{section}
\newtheorem{thm}{Theorem}[section]
\newtheorem{prop}[thm]{Proposition}
\newtheorem{lem}[thm]{Lemma}
\newtheorem{coro}[thm]{Corollary}
\theoremstyle{definition}
\newtheorem{defx}[thm]{Definition}
\newtheorem{rem}[thm]{Remark}

\newcommand{\R}{I\!\!R}

\DeclareMathSizes{12}{13}{10}{10}

\title{Groupoid Characterization of Partial Algebras on Sobolev Spaces}
\author{{\emph{N. O. Okeke \;  Email: nickokeke@dui.edu.ng, }}
\\ Physical and Mathematical Sciences, Dominican University, Ibadan \\
\emph{M. E. Egwe \; Email: murphy.egwe@ui.edu.ng}
\\ Department of Mathematics, University of Ibadan, Ibadan, Nigeria}
\date{\vspace{-5ex}}

\begin{document}
\date{}
\maketitle
\begin{abstract}
The $L^p$-spaces, with $p \not = \infty$, form a partial algebra $(L^p(\Omega), \Gamma, \cdot)$ with pointwise multiplication of functions. The Sobolev spaces $W^{k,p}(\Omega)$, delineated by weak derivatives as subspaces of $L^p$-spaces is shown to contain the partial algebra $(L^p(\Omega), \Gamma, \cdot)$ generalized by the partial action of the smooth algebra $\mathscr{K}(\Omega)$ by convolution on the Banach spaces $L^p(\Omega)$. We characterised the Sobolev space $W^{k,p}(\Omega)$, invariant under $\mathscr{K}(\Omega)$ partial action, using Lie groupoid framework, and study the partial algebra as defining the partial dynamical systems on the $L^p$-space associated with the weak differential operators. The locally convex partial $^*$-algebra $(L^p(\Omega), \Gamma, \cdot,^*)$ defines the stable local flows coinciding with local bisections of the Lie groupoid. The unitary representation of resulting Lie groupoid $\mathscr{W} \rightrightarrows W^{k,p}(\Omega)$ on the associated Hilbert bundle demonstrates the simplification achieved by the characterisation.
\end{abstract}
\let\thefootnote\relax{\emph{Key words and phrases:} Sobolev spaces, regular distributions, smoothing algebra, pseudogroup, Lie groupoid, partial $^*$-algebras, locally convex partial $^*$-algebras.}

\section{Introduction}
This paper is the twin of \cite{OkekeEgwe2023}, where the introductory materials are given. Some relevant information will be repeated in the present paper for continuity. In the first paper we showed how the Sobolev spaces come up from our desire to analyse the solutions of certain differential equations. The purpose of this paper is to display the partial algebras defined by the pointwise product on $L^p(\Omega)$ using the Lie groupoid framework. The paper aims at simplifying the understanding of partial algebras and partial $^*$-algebras and their dynamical representations. To achieve this: Sequences will be generalised to nets based on differentiability criteria placed on the functions defined on the manifold of our interest $\R^n$ and made to act on the Banach space $L^p(\Omega)$.

Given a pair of functions $f,g \in L^p(\Omega)$, the relation $\Gamma = \{(f,g) \in L^p(\Omega) \times L^p(\Omega) : f\cdot g \in L^p(\Omega)\}$ which defines a partial algebra $\Gamma \subset L^p \times L^p$ will be located or embedded in the Sobolev space $W^{k,p}(\Omega)$ using the proximate identities by smooth kernels constituting the nets. The partial algebra $\Gamma$ will then be given a Lie groupoid characterisation and studied as such.

The concept of a smooth net $\phi_\varepsilon$ follows from continuous net when each component is non-negative, uniformly bounded and continuous, integrable, and infinitely differentiable. That is, each $\phi$ is of the class $C^\infty(\R^n)$ supported on a compact subset of $\R^n$. A smooth algebra is defined with smooth nets as follows.
\begin{defx}
A smooth algebra is a collection of smoothing nets defined on the open subset $\Omega \subset \R^n$ having the following properties (i) $\phi_\varepsilon \in C^\infty_o(\R^n)$, (ii) $\text{supp}(\phi_\varepsilon) \subset \bar{B}(0,\varepsilon)$, (iii) $\displaystyle \int \phi_\varepsilon(x)dx = 1$, and closed under both pointwise and convolution products. It is a smooth algebra of units or approximate identities denoted by $\mathscr{K}(\Omega)$.
\end{defx}
\begin{rem}
(1) Any smooth function $\phi : \R^n \to [0,\infty)$ supported on a unit ball $B(1)$, which can be supported on any closed ball by proper modification such that $\displaystyle \int \phi d\mu^n = 1$, defines an approximate identity by suitable adjustment \[\phi_\varepsilon(x) = \varepsilon^{-n}\phi(\frac{x}{\varepsilon}). \eqno{(1.0)}  \]
Examples of such smooth kernels are given in \cite{AlGwaiz92}, \cite{Jost2005} and \cite{Grosseretal2001}.\\
(2) The smooth net is different from the delta net and the product nets as defined and explained in \cite{Grosseretal2001}. It is a natural tool intrinsically related to the local dynamical system on the Banach subspace $L^p(\Omega)$, which we study by its action. The action results in fibre structure which carries a distributional properties  through a weak differential structure. \\
(3) The employment of the smooth net in the canonical definition of the generalized algebra of distributions in section 3.3.2 of \cite{Grosseretal2001} points to its importance as intrinsic tool for analysis on smooth manifolds.
\end{rem}
\begin{defx}
Denote the weak derivative of $f$ in $x^i$-direction with $D_if$, and $Df = (D_1f, \cdots, D_nf)$ when $f$ has a weak derivative for all $i = 1,\cdots, n$. Then with the multi-index $\alpha := (\alpha_1, \cdots, \alpha_n), \alpha_i \geq 0, (i = 1, \cdots, n)$, and the norm given as $ |\alpha| := \overset{n}{\underset{i=1}\sum}\alpha_i > 0$, we have \[ D_\alpha \varphi := \left(\frac{\partial}{\partial x^1}\right)^{\alpha_1} \cdots \left(\frac{\partial}{\partial x^n}\right)^{\alpha_n}\varphi = \frac{\partial^{|\alpha|}\varphi}{\partial x_1^{\alpha_1} \cdots \partial x_n^{\alpha_n}} \; \text{for } \varphi \in C^{|\alpha|}(\Omega). \]
Then a function $u \in L^1_{loc}(\Omega)$ is called the $\alpha$-th weak derivative of $f$ if \[ \int_\Omega \varphi udx = (-1)^{|\alpha|}\int_\Omega f D_\alpha\varphi dx, \; \text{for all } \varphi \in C^\infty_c(\Omega) \subset C^{|\alpha|}_o(\Omega). \]
\end{defx}
\begin{defx}
For $k \in \mathbb{N}, 1 \leq p \leq \infty$, the Sobolev spaces are defined as the sets $W^{k,p}(\Omega) := \{f \in L^p(\Omega) : D_\alpha f \; \text{exists and is in } L^p(\Omega) \; \forall \; |\alpha| \leq k \}$. The norm on the spaces is defined as \[||f||_{W^{k,p}(\Omega)} = ||f||_{k,p} := \left(\underset{|\alpha| \leq k}\sum \int_\Omega |D_\alpha f|^p \right)^{\frac{1}{p}} \; \text{for } 1 \leq p < \infty. \] and \[||f||_{W^{k,\infty}(\Omega)} = ||f||_{k,\infty} := \underset{|\alpha| \leq k}\sum \underset{x \in \Omega}{\text{ess}\sup}|D_\alpha f(x)|. \]
\end{defx}
Given the Sobolev space $W^{k,p}(\Omega)$ as the space of functions whose weak derivatives up to order $k$ are in $L^p(\Omega)$, the space was studied using the slice theorem in \cite{OkekeEgwe2023}. It is used here to study the partial algebra $\Gamma$ on $L^p(\Omega)$ using groupoid framework. The unitary representation of these partial algebras on the Hilbert bundle of the Lie groupoid follows from \cite{OkekeEgwe20202} and showcases the simplification achieved in comparison to the representations of these algebras in \cite{Ekhaguere2007}.

The definition of the partial action of the smooth algebra $\mathscr{K}(\Omega)$ by convolution on $L^p(\Omega)$ is given in \cite{OkekeEgwe2023}. Since the action is closed on the Sobolev space $W^{k,p}(\Omega)$, the following result, proved differently in \cite{OkekeEgwe2023}, shows the orbit of the partial action of $\mathscr{K}(\Omega)$ by its nets in $W^{k,p}(\Omega)$.
\begin{prop}
Given $\Omega' \subset \subset \Omega$ and $\varepsilon < dist(x, \partial \Omega)$ then $\phi_\varepsilon(f) = f_\varepsilon$ is an orbit of the $\mathscr{K}(\Omega)$-action on $W^{k,p}(\Omega) \subset L^p(\Omega)$.
\end{prop}
\begin{proof}
The proof follows from the modification of smooth functions on a closed ball to approximate identities described above and the definition of the convolution action of $\mathscr{K}(\Omega)$ by its nets on $L^p(\Omega)$:
$(\phi_\varepsilon)$ as follows \begin{equation} \mathscr{K}(\Omega) \times L^p(\Omega) \to L^p(\Omega), \phi_\varepsilon(f) := \frac{1}{\varepsilon^n}\int_\Omega \phi(\frac{x-y}{\varepsilon})f(y)dy = f_\varepsilon(x).\end{equation}
Thus, $f_\varepsilon(x)$ is an orbit of the partial $\mathscr{K}(\Omega)$-action on $W^{k,p}(\Omega) \subset L^p(\Omega)$.
\end{proof}

\subsection{Partial Algebra on Sobolev Space}
The closure of Sobolev spaces $W^{k,p}(\Omega)$ under weak derivative and its invariance under the partial action of the smooth algebra $\mathscr{K}(\Omega)$ by its nets imply the existence of solutions for every system of partial differential equations defined on a closed and compact subset $\Omega' \subset \subset \Omega$. The spaces have equivalence defined on them by the smoothing algebra $\mathscr{K}(\Omega) \subset C_c^\infty(\Omega)$ made up of contractive transformations on a Banach space as shown in \cite{OkekeEgwe2023}. Since the fixed point of a contraction depends smoothly on the contraction, the contraction and its fixed point are equivariantly transformed or preserved under equivariant transformations. This gives rise to the idea of persistence of transverse fixed points. This idea will be explored in application of this work in solving \emph{Split Feasibility Problem} (SFP) in a different paper.

The use of convolution action of the smooth algebra $\mathscr{K}(\Omega)$ to study the partial algebra $\Gamma$ within the Sobolev space $W^{k,p}(\Omega)$ is based on the understanding that a simple condition on the linear part of a map at a fixed point guarantees the persistence of the fixed point under perturbation, as noted in \cite{HasselblattKatok}. Since $C_c^\infty(\Omega)$ is dense in $L_{loc}^1(\Omega)$, the weak derivative of a function is unique up to an equivalence defined pointwise almost everywhere. Hence, the weak derivative of a continuously differentiable function agrees with the pointwise derivative when it exits. But the existence of the two types of derivatives is not equivalent pointwise almost everywhere.

The embedding of the partial algebra $\Gamma$ into $W^{k,p}(\Omega)$ helps the replacement of the pointwise derivative of a product $f\cdot g$ with its weak derivative. So, the weak derivative of a product $f\cdot g$ exists even when the pointwise derivative does not exist, which clarifies the idea of embedding. As noted above, the partial action of $\mathscr{K}(\Omega)$ by its nets introduces the notion of transversality by adding differentiability to the algebra. This means that small perturbations to the functions do not destroy their well-behaved property. According to \cite{HasselblattKatok}, this guarantees a good linear approximations of any map near a given point. Hence, the closure of the algebra under small perturbations, which are maps on $\Gamma$ by $\mathscr{K}(\Omega)$, produces invertible maps by transversality, with the inverses composed with strongly contracting maps coming from the perturbation.This makes our algebra remain the fixed points of the contractions and the stability of the partial flows. The partial dynamical system is given in \cite{OkekeEgwe2023} as the quadruple $(L^p(\Omega), \mathscr{K}(\Omega), W^{k,p}(\Omega), \phi_\varepsilon)$. This makes a net $\phi_\varepsilon \subset \mathscr{K}(\Omega)$ acting on $W^{k,p}(\Omega) \subset L^p(\Omega)$ a partial dynamical system on the $L^p$-spaces. The convergence of local flow to the partial algebra $\Gamma$ is the major result proved within the groupoid framework.
\begin{thm}
The partial algebra $\Gamma$ is fixed or stable under the local perturbations of $L^p(\Omega)$ by $\mathscr{K}(\Omega)$ contained in the Sobolev space $W^{k,p}(\Omega)$.
\end{thm}
\begin{proof}
As shown in \cite{OkekeEgwe2023}, the convolution action of the smooth algebra $\mathscr{K}(\Omega)$ by its nets creates the equivalence which makes $W^{k,p}(\Omega)$ the space of $\mathscr{K}(\Omega)$-generalized quantities. Thus, we have a local equivalence $W^{k,p}(\Omega) \equiv \mathscr{K}(\Omega) \times L^p(\Omega)$, making $W^{k,p}(\Omega)$ closed subspaces under local perturbations of $L^p(\Omega)$ by $\mathscr{K}(\Omega)$.
\end{proof}
The framework of the Lie groupoid will be used to characterise the equivalence on the Sobolev spaces and the relation defined by the partial product operation on the $L^p$-spaces. Thus, the resulting Lie groupoid algebra will represent the locally convex partial $^*$-algebra. The definition of partial algebras, partial $^*$-algebras, and locally convex partial $^*$-algebras will now be streamlined to the Sobolev spaces. The following proposition follows on the definition of partial algebra in \cite{Ekhaguere2007}.
\begin{prop}
The triple $(L^p(\Omega), \Gamma, \cdot )$ is a partial algebra, with the pointwise multiplication of functions $f \cdot g$ partially defined on the linear space $L^p(\Omega)$.
\end{prop}
\begin{rem}
To embed the partial algebra $(L^p(\Omega),\Gamma, \cdot)$ in $W^{k,p}(\Omega)$, the Sobolev space of all the $k$th weakly differentiable functions of $L^p(\Omega)$, where $\Omega \subset \R^n$, we show that the relation satisfies the definition of the Sobolev space.
\end{rem}
\begin{prop}
The partial algebra $(L^p(\Omega),\Gamma, \cdot)$ is contained in the Sobolev space $W^{k,p}(\Omega)$, a subspace of the linear space $L^p(\Omega)$.
\end{prop}
\begin{proof}
Given a pair of functions $(f,g)$ in the relation $\Gamma = \{(f,g) \in L^p(\Omega) \times L^p(\Omega) : f\cdot g \in L^p(\Omega)\}$, the functions and the product $f\cdot g$ are elements in $W^{k,p}(\Omega)$ by the partial action of $\mathscr{K}(\Omega)$ on $L^p(\Omega)$. This follows from the definition of Sobolev space. The regularization involved in the definition of a Sobolev space ensures the continuity of $f\cdot g$ whenever $f\cdot g \in L^p(\Omega)$.

On the other hand, any pair of functions $f,g \in L^p(\Omega)$ with product $f\cdot g \not \in L^p(\Omega)$, the smoothing may make the product $f\cdot g \in W^{k,p}(\Omega)$ by its definition. Also, any pair of functions $f,g \in L^p(\Omega)$ with product $f\cdot g \in W^{k,p}(\Omega) \; \implies f\cdot g \in L^p(\Omega) = W^{0,p}(\Omega)$. This means that the partial algebra $(L^p(\Omega), \Gamma, \cdot)$ with the partial action of $\mathscr{K}(\Omega)$, is contained in the Sobolev space $W^{k,p}(\Omega)$. That is, $\Gamma \times \mathscr{K}(\Omega) \subset W^{k,p}(\Omega) \times W^{k,p}(\Omega)$.
\end{proof}
Let the relation $\Gamma = \{(f,g) \in L^p(\Omega) \times L^p(\Omega) : f\cdot g \in L^p(\Omega)\}$ defined on the linear space $L^p(\Omega)$ be given on the Sobolev space $W^{k,p}(\Omega)$ as $\mathscr{W} = \{(f,g) \in W^{k,p}(\Omega) \times W^{k,p}(\Omega) : f\cdot g \in W^{k,p}(\Omega)\}$. It follows that $\Gamma \subseteq \mathscr{W}$ on $L^p(\Omega)$. Then we have the following groupoid characterization of the partial algebra.
\begin{prop}
The partial algebra $(L^p(\Omega), \Gamma, \cdot)$ is equivalent to the groupoid $\mathscr{W} \rightrightarrows W^{k,p}(\Omega)$ defined by the relation $\mathscr{W} = \{(f,g) \in W^{k,p}(\Omega) \times W^{k,p}(\Omega) : f \cdot g \in W^{k,p}(\Omega)$ on the linear space $W^{k,p}(\Omega)$.
\end{prop}
\begin{proof}
The proof follows from Proposition 1.2 of \cite{OkekeEgwe20202}. In this case, the elements of $W^{k,p}(\Omega)$ are in $\mathscr{W}$ by pairing them with the identities $\delta_x \in W^{k,p}(\Omega)$, such that $(\delta_x,f) \in \mathscr{W}$. We note that the groupoid is closed under net (or generalized sequence) convergence, since $W^{k,p}(\Omega)$ is a Banach space.
\end{proof}
\begin{rem}
The space of arrows of the groupoid is simply denoted by $\mathscr{W}$ as resulting from the relation $\Gamma$ in $W^{k,p}(\Omega)$. The linear conditions make the sets $\langle f,- \rangle, \langle -,g \rangle$ of right and left \emph{multipliers} respectively, of any element of the linear space $L^p(\Omega)$ into linear subspaces determined by the relation $\Gamma = \{(f,g) : f\cdot g \in L^p(\Omega) \}$ defined by the partial product. In terms of the groupoid, an arrow $(f,g) \in \mathscr{W}$ determines two linear subspaces $\mathscr{W}(f,-), \mathscr{W}(-,g)$ by its target and source maps $(t,s) : \mathscr{W} \to W^{k,p}(\Omega)$ which are the target and source fibres. The two subspaces are related as follows \[s^{-1}(g) = \mathscr{W}(-,g) = \{(f,g) \in \mathscr{W} : f \in \langle -,g \rangle \};\] \[ t^{-1}(f) = \mathscr{W}(f,-) = \{(f,h) \in \mathscr{W} : h\in \langle f,- \rangle \}.\] This means that the \emph{right multipliers} define the target fibre, and the \emph{left multipliers} define the source fibre. This formulation is extended to a partial $^*$-algebra defined in \cite{Ekhaguere2007} as follows.
\end{rem}
Based on the definitions of \cite{Ekhaguere2007}, given the partial algebra $(L^p(\Omega), \Gamma, \cdot)$, the quadruplet $((L^p(\Omega), \Gamma, \cdot,^*)$ is a \emph{partial $^*$-algebra} or an \emph{involutive partial algebra}, such that $L^p(\Omega)$ is an \emph{involutive} linear space with \emph{involution} $^*$, $(g^*,f^*) \in \Gamma$ whenever $(f,g) \in \Gamma$ and then $(f\cdot g)^* = g^*\cdot f^*$. These hold whether $L^p(\Omega)$ is real or complex valued.

Following also from \cite{Ekhaguere2007}, a \emph{partial subalgebra} (respectively \emph{partial $^*$-subalgebra}) is a subspace (respectively a $^*$-invariant subspace) $\mathcal{B}$ of $L^p(\Omega)$ such that $f\cdot g \in \mathcal{B}$ whenever $f,g \in \mathcal{B}$ and $(f,g) \in \Gamma$.
\begin{prop}
Given the partial $^*$-algebra $(L^p(\Omega), \Gamma, \cdot,^*)$, there is a corresponding groupoid $\mathscr{W} \rightrightarrows W^{k,p}(\Omega))$ defined by the equivalent relation $\mathscr{W} = \{(f,g) \in W^{k,p}(\Omega) \times W^{k,p}(\Omega) : f\cdot g \in W^{k,p}(\Omega)\}$ over $W^{k,p}(\Omega)$, such that the compatibility of the partial product $\cdot$ and the involution $^*$ in $L^p(\Omega)$ implies $(f\cdot g)^* = g^*\cdot f^*$.
\end{prop}
\begin{proof}
This follows immediately from the above groupoid characterization of the partial algebra $(L^p(\Omega), \Gamma, \cdot)$, with the addition of the involutive map $^* : L^p(\Omega) \to L^p(\Omega)$ which is compatible with the partial product $\cdot : \Gamma \to L^p(\Omega)$ defining the relation $\Gamma$ on $L^p(\Omega)$.

On the groupoid $\mathscr{W}$, we identify the involutive map $* : W^{k,p}(\Omega) \to W^{k,p}(\Omega)$ with the inverse map $i : W^{k,p}(\Omega) \to W^{k,p}(\Omega)$, then $f^*(x) = f^{-1}(x) = \frac{1}{f(x)}$ for $x \in \Omega$ and $(f\cdot g)^* = (f\cdot g)^{-1} = g^{-1}\cdot f^{-1} = g^*\cdot f^*$. Thus, $(g^*,f^*) \in \mathscr{W}$ whenever $(f,g) \in \mathscr{W}$.

As we noted above, the units in the groupoid $\mathscr{W}$ are the delta functions $\delta(x)$ which are zero function in $L^p(\Omega)$, which means that $f \equiv 0 \in L^p(\Omega)$ belongs to equivalent class of $\delta \in W^{k,p}(\Omega)$, identities on $W^{k,p}(\Omega)$. So, $(\delta,f) \in \mathscr{W}$ implies that $s(\delta,f) = t(\delta,f) = f$. By definition of pointwise multiplication on the Sobolev space, the unit $(\delta,f)_x \; \implies \; \langle f,\delta \rangle_x = \delta_x(f) = f(x)$. Hence, the objection map is defined using the delta function as follows $o : W^{k,p}(\Omega) \to \mathscr{W}, f \mapsto (\delta,f)$. The involution of an identity $\delta \in W^{k,p}(\Omega)$ is $\delta^* = \frac{1}{\delta} = \delta$, and involution of an arrow in $\mathscr{W}$ is given as $(\delta,f)^* = (\delta,f)^{-1} = f^{-1}\cdot \delta^{-1} = (f^*, \delta^*) = \delta(f^{-1}) = (f^*, \delta) = f^*$.
\end{proof}
\begin{defx}
The left and right \emph{multipliers} of $L^p(\Omega)$ are respectively given as:  \[\langle -,L^p(\Omega)\rangle = \{f \in L^p(\Omega) : (f,g) \in \mathscr{W}, \; \forall \; g \in L^p(\Omega) \} \] \[\langle L^p(\Omega),- \rangle = \{g \in L^p(\Omega) : (f,g) \in \mathscr{W}, \; \forall \; g \in L^p(\Omega)\}.\]
The left multipliers $\langle \cdot,g \rangle \subset s^{-1}(g)$ of $g$ and right multipliers $\langle f,\cdot \rangle \subset t^{-1}(f)$ of $f$ are equal to $L^p(\Omega)$ whenever $f \in \langle L^p(\Omega),- \rangle, g \in \langle -,L^p(\Omega)\rangle$. A function $f \in L^p(\Omega)\rangle$ is a universal multiplier if it is both left and right multiplier of $L^p(\Omega)\rangle$. In the case of the groupoid $\mathscr{W}$, if $g$ is a left multiplier of $L^p(\Omega)$ and $f$ is a right multiplier of $L^p(\Omega)$, then $L^p(\Omega) \subset \mathscr{W}(-,g)$ and $L^p(\Omega) \subset \mathscr{W}(f,-)$.
\end{defx}
For a function to define a weak derivative it must be finite on the neighbourhood of a given point. Most especially, on the neighbourhood of the zero vector $0$. This means that it must be locally integrable on the open subset $\Omega$. Since every $p$-integrable function is locally integrable we have $L^p(\Omega) \subset L^1_{loc}(\Omega)$. The following remarks are about the smooth multiplier of $L^p(\Omega)\rangle$.\\
\begin{rem}
(1) The compactly supported smooth functions are locally integrable; hence, they are used to define a smooth algebra $\mathscr{K}(\Omega)$.\\
(2) The smooth algebra $\mathscr{K}(\Omega)$ is contained in the set of universal multipliers of $L^p(\Omega)$. Therefore, they are used to define weak differential structure on the Lie groupoid $\mathscr{W} \rightrightarrows W^{k,p}(\Omega)$.
\end{rem}
\begin{prop}
The weak derivative is compatible with the pointwise multiplication and introduces a weak differential structure on the relation $\Gamma$ in $W^{k,p}(\Omega)$.
\end{prop}
\begin{proof}
The compatibility of the weak derivative and the pointwise multiplication is already shown in the embedding of the algebra $\Gamma$ in $W^{k,p}(\Omega)$, the arrows of the groupoid $\mathscr{W}$ are defined by the relation. Hence, the partial smooth action on the relation $\Gamma$ gives smoothness property to the groupoid $\mathscr{W} = \Gamma \times \mathscr{K}(\Omega)$ on the Sobolev space $W^{k,p}(\Omega)$. It introduces a differential structure on the relation $\mathscr{W} \rightrightarrows W^{k,p}(\Omega)$, making it a Lie groupoid.
\end{proof}
Subsequently, for a compact subset $\Omega' \subset \Omega$, the pointwise multiplication is defined for all pair $f,g \in L^p(\Omega')$ such that the zero order Sobolev space is $W^{0,p}(\Omega') = L^p(\Omega')$. For any other nonzero finite order $k < \infty$ we have proper containment $W^{k,p}(\Omega) \subset L^p(\Omega)$. It follows that $W^{k,p}(\Omega)$ is invariant under the action of the groupoid $\mathscr{W}$, just as $L^p(\Omega)$ is invariant under the iteration of a left multiplier $f \in \langle -,L^p(\Omega) \rangle$ (respectively a right multiplier $g  \in \langle L^p(\Omega),- \rangle$). This gives rise to the proposition.
\begin{prop}
$\langle -,L^p(\Omega)\rangle$ and $\langle L^p(\Omega),- \rangle$ are left and right ideals of the partial $^*$-algebra $(L^p(\Omega),\mathscr{W},\cdot,^*)$.
\end{prop}
\begin{proof}
By definition $\langle -,L^p(\Omega)\rangle$ (respectively $\langle L^p(\Omega),- \rangle$) is closed (or invariant) under left (respectively right) multiplication by $L^p(\Omega)$; that is, $L^p(\Omega) \times \langle -,L^p(\Omega)\rangle  \to \langle -,L^p(\Omega)\rangle$ and $\langle L^p(\Omega),- \rangle \times L^p(\Omega) \to \langle L^p(\Omega),- \rangle$. Thus, the restriction of the partial multiplication on these subsets makes it a full multiplication or an algebra.
\end{proof}
The above result is equivalently given in terms of the target $t$-fibre and the source $s$-fibre in the groupoid framework as follows.
\begin{prop}
The target $t$-fibre and the source $s$-fibre of the Lie groupoid $\mathscr{W} \rightrightarrows W^{k,p}(\Omega)$ are invariant under the action of the Lie groupoid or its arrows $\mathscr{W}$.
\end{prop}
\begin{proof}
The arrows of the Lie groupoid are modelled on the partial algebra $\Gamma$. Thus the invariance of the left and right multipliers under the partial algebra $\Gamma$ translates to the action of the Lie groupoid $\mathscr{W}$. The smoothing effect of the algebra $\mathscr{K}(\Omega)$ makes the algebra $\mathscr{K}(\Omega)$ the ideal of the Banach space $L^p(\Omega)$ and the Sobolev space $W^{k,p}(\Omega)$.

The Lie groupoid $\mathscr{W}$ represents the partial algebra $\Gamma$ with the partial action of $\mathscr{K}(\Omega)$. Thus, the action of $\mathscr{K}(\Omega)$ falls under the action of the vertex group of the Lie groupoid $\mathscr{W} \rightrightarrows W^{k,p}(\Omega)$ given as: \[ \mathscr{W}(f,-) \times \mathscr{K}(\Omega) \to \mathscr{W}(f,-),\;  f \star \varphi_\varepsilon = f_\varepsilon.  \]
This results in $t,s$-fibres of equal dimension at each point of the base space $W^{k,p}(\Omega)$. Since for each pair $(f,g) \in W^{k,p}(\Omega) \times W^{k,p}(\Omega)$, $f\cdot g \in W^{k,p}(\Omega)$, then $$(f_\varepsilon, g) \to (f,g) \equiv (f,g_\varepsilon) \to (f,g),$$ giving rise to the left partial action of $\mathscr{K}(\Omega)$ on the source fibre \[\mathscr{K}(\Omega) \times \mathscr{W}(-,g) \to \mathscr{W}(-,g),\; \varphi_\varepsilon \star g = g_\varepsilon.\]  Thus, both $t$-fibre and $s$-fibre are invariant of the partial action of $\mathscr{K}(\Omega)$. Given $\mathscr{K}(\Omega) \subseteq \mathscr{W}(f,f)$, the action of the isotropy or vertex group of the Lie groupoid $\mathscr{W}$, contains the local dynamics of the partial algebra $\Gamma$.
\end{proof}
Following the definition of a locally convex partial $^*$-algebra in \cite{Ekhaguere2007}, we have the quadruplet $(L^p(\Omega),\Gamma,\cdot,\tau)$ as a \emph{locally convex partial algebra} ( and respectively the quintuplet $(L^p(\Omega),\Gamma,\cdot,^*,\tau)$ as a \emph{locally convex partial $^*$-algebra}) comprising a partial algebra $(L^p(\Omega),\Gamma,\cdot)$ (respectively a partial $^*$-algebra $(L^p(\Omega),\Gamma,\cdot,^*)$) and a Hausdorff locally convex topology $\tau$ and $(L^p(\Omega),\tau)$ is a locally convex space and the maps $g \mapsto f\cdot g$ and $h \mapsto h \cdot f$ are continuous for every $f \in \langle -,g \rangle$ and $h \in \langle f,- \rangle$ (respectively the maps $g \mapsto g^*, g \mapsto f\cdot g$ and $f \mapsto h \cdot f$ are continuous for every $g \in L^p(\Omega), f \in \langle -,g \rangle$ and $h \in \langle f,- \rangle$).
\begin{prop}
The locally convex partial $^*$-algebra $(L^p(\Omega),\Gamma,\cdot,^*,\tau)$ gives rise to the locally convex groupoid $\mathscr{W} \rightrightarrows W^{k,p}(\Omega)$ such that the space of arrows  $\mathscr{W}$ is given by the relation $\mathscr{W} = \{(f,g) \in W^{k,p}(\Omega) \times W^{k,p}(\Omega) : f\cdot g \in W^{k,p}(\Omega)\}$ on $W^{k,p}(\Omega)$.
\end{prop}
\begin{proof}
The proof follows from above given that $(L^p(\Omega),\tau)$ is a Hausdorff locally convex topological space for all $p \geq 1$; and $\mathscr{W}$ as a closed subspace of $L^p(\Omega) \times L^p(\Omega)$ endowed with the locally convex topology induced from $\mathcal{D}'(\Omega) \times \mathcal{D}'(\Omega)$, and the  continuity of the \emph{partial multiplication} $\cdot$ that defines the relation $\mathscr{W}$ preserves local convexity; the continuity of involution $^*$ is by definition; also, the defining maps of the groupoid: the target and source maps $t,s : \mathscr{W} \to W^{k,p}(\Omega)$, the inverse $i : \mathscr{W} \to \mathscr{W}$, and the composition of arrows $m : \mathscr{W} \times \mathscr{W} \to \mathscr{W}$ are all continuous maps by definition of the partial product. Thus, $\mathscr{W} \rightrightarrows W^{k,p}(\Omega)$ is a locally convex topological groupoid characterising the locally convex partial $^*$-algebra.
\end{proof}
The realised groupoid satisfies the following definitions of the topological groupoid and locally convex topological groupoid.
The groupoid $\mathscr{W} \rightrightarrows W^{k,p}(\Omega)$ is a topological groupoid since its set of morphisms $\mathscr{W}$ and set of objects $W^{k,p}(\Omega)$ are topological spaces, and the composition $m : \mathscr{W} \times \mathscr{W} \to \mathscr{W}$, source and target $t,s : \mathscr{W} \to L^p(\Omega)$, objection $o : L^p(\Omega) \to \mathscr{W}$, and inversion $i : \mathscr{W} \to \mathscr{W}$ maps are continuous, with the induced topology on $\mathscr{W}^{(2)}$ from $\mathscr{W} \times \mathscr{W}$.

By a modification of the definition of locally convex Lie groupoid given in (\cite{SchmedingWockel2014}, 1.1) and adapted in \cite{OkekeEgwe20202}, it  clearly follows that $\mathscr{W} \rightrightarrows W^{k,p}(\Omega)$ which is a groupoid over $W^{k,p}(\Omega)$ with the source and target $t,s : \mathscr{W} \to W^{k,p}(\Omega)$ projections is a locally convex groupoid. Since $\mathscr{W}$ is a locally convex (and locally metrizable) topological groupoid over $W^{k,p}(\Omega)$ given that (i) $W^{k,p}(\Omega)$ and $\mathscr{W} \rightrightarrows W^{k,p}(\Omega)$ are locally convex spaces; (ii) the topological structure of $\mathscr{W}$ makes $s$ and $t$ continuous; i.e. local projections; and (iii) the partial composition $m : \mathscr{W} \times_{s,t}\mathscr{W} \to \mathscr{W}$, objection $o : W^{k,p}(\Omega) \to \mathscr{W}$, and inversion $i : \mathscr{W} \to \mathscr{W}$ are continuous maps.
\begin{thm}
The groupoid $\mathscr{W} \rightrightarrows W^{k,p}(\Omega)$ resulting from the locally convex partial $^*$-algebra $(L^p(\Omega),\mathscr{W},\cdot,^*,\tau)$ is a Lie groupoid modelled on the locally convex space $W^{k,p}(\Omega)$.
\end{thm}
\begin{proof}
Following the argument in \cite{OkekeEgwe20202}, with a background from \cite{OmokriHideki97}, the open subspace of right multipliers $\langle f,- \rangle$ and the target fibre $\mathscr{W}(f,-)$ (or its isomorphic source fibre $\mathscr{W}(-,g)$ and left multipliers $\langle -,g \rangle$) of the groupoid are considered (weak derivative) equivalent given that the groupoid was constructed from the relation $\Gamma$ on $L^p(\Omega)$ by a partial action of $\mathscr{K}(\Omega)$. As noted above, these subspaces are maximal when $f \in \langle-,L^p(\Omega) \rangle$ (respectively $g \in \langle L^p(\Omega),- \rangle$). The maximality points to the space of test functions $\mathcal{D}(\Omega)$, showing that the connected components or the convergent nets are Lie pseudogroups. The proof of the result now follows from \cite{OkekeEgwe20202}.
\end{proof}

\section{Convolution Algebra of Sobolev groupoid}
The coming together of the partial algebra $\Gamma$ and the convolution action of the smooth algebra $\mathscr{K}(\Omega)$ on $L^p(\Omega)$ gives rise to the definition of the convolution algebra of the Lie groupoid $\mathscr{W}$. The relationship between the partial  algebra $\Gamma \subset L^p(\Omega) \times L^p(\Omega)$ and the Sobolev space $W^{k,p}(\Omega)$ is further understood using the concept of proper action and slice which we have demonstrated as applicable in the earlier paper \cite{OkekeEgwe2023}. Using the notion of a proper $G$-space as defined by Palais \cite{Palais60}, we consider $W^{k,p}(\Omega)$ a proper $\mathscr{K}(\Omega)$-space; for given any $U$--an open set of $L^p(\Omega)$--it is also a $\mathscr{K}(\Omega)$-space in the sense of the partial convolution actions of $\mathscr{K}(\Omega)$ by its nets. Thus, we can realise $\mathscr{K}(\Omega)U = W^{k,p}(\Omega)$ as a proper $\mathscr{K}(\Omega)$-space. In this case, we consider the nets $f_h \in \mathscr{K}(\Omega)$ as fixing $f \in L^p(\Omega)$, then we have a similar structure of transitivity, with $H = \underset{f \in L^p(\Omega)}\bigcap \mathscr{K}(\Omega)_f$, where $\mathscr{K}(\Omega)_f$ are set of nets in $\mathscr{K}(\Omega)$ converging to $f$, or whose action preserve $f\in L^p(\Omega)$. They are equivariant nets.

The quotient $\mathscr{K}(\Omega)/H$ has a transitive action on $L^p(\Omega)$. The definition of a local cross section by convergent nets is derived from this formulation, for just as a locally compact Lie group $G$ acts properly on a space $M$ by its closed or compact subsgroups, the algebra $\mathscr{K}(\Omega)$ acts properly on $L^p(\Omega)$ by its (equivariant) convergent nets $f_h \to f$ at $f \in L^p(\Omega)$. The local cross section associated to the equivariant nets is also helpful in defining the system of invariant Haar measures on the emergent Lie groupoid $\mathscr{W} \rightrightarrows W^{k,p}(\Omega)$. The proper action of $\mathscr{K}(\Omega)$ by its smooth nets is related to the smooth local cross sections in the sense that the convergence of the nets in $\mathscr{K}(\Omega)$ leads to convergence in the proper $\mathscr{K}(\Omega)$-space. The following result connects the local section of the partial $\mathscr{K}(\Omega)$-action to the local bisections of $\mathscr{W}$.
\begin{thm}
The partial action of $\mathscr{K}(\Omega)$ creates local flows which coincides with the local bisections of the Lie groupoid $\mathscr{W}$.
\end{thm}
\begin{proof}
By the partial action of the smooth algebra $\mathscr{K}(\Omega)$ by its nets, we have shown the convergence of the local smooth flows $\varphi_\varepsilon(f) = f_\varepsilon \to f$. By the continuity of the pointwise multiplication $f\cdot g$, the smoothness passes unto the product $f_\varepsilon \cdot g \to f\cdot g$, and hence, invariant under the partial algebra $\Gamma \subset W^{k,p}(\Omega) \times W^{k,p}(\Omega)$. \\
The partial action trivializes the source map $s: \mathscr{W} \to W^{k,p}(\Omega)$; hence, corresponds to local bisections $\sigma : W^{k,p}(\Omega) \to \mathscr{W}$, such that $$t \circ \sigma : W^{k,p}(\Omega) \to W^{k,p}(\Omega)$$ is a diffeomorphism. That $t \circ \sigma$ is diffeomorphism follows from the smoothing of the arrows $(f,g) \in \Gamma$. Thus, the local flows corresponding to the partial algebra $\Gamma$ in the Lie groupoid $\mathscr{W} \rightrightarrows W^{k,p}(\Omega)$ are given by the local bisections of $\mathscr{W}$ since the diffeomorphism $t \circ \sigma$ follows the relation $\Gamma$ defining the arrows of $\mathscr{W}$.
\end{proof}
The proper action of $\mathscr{K}(\Omega)$ on $L^p(\Omega)$, embedding the partial algebra $\Gamma$ into $W^{k,p}(\Omega)$ as its stable or invariant subspace, makes the nets $f_\varepsilon \in W^{k,p}(\Omega)$ converges to $f \in W^{k,p}(\Omega)$ as $W^{k,p}(\Omega) \times W^{k,p}(\Omega)$ converge to the partial algebra $\Gamma$. The action of $\mathscr{K}(\Omega)$ on $\Gamma$, which attaches differential structure to it, is similar to line bundles and related to the exterior space $\Omega(\mathscr{W})$ of the Lie groupoid $\mathscr{W}$. Thus, the densities, which are the same as differential forms with orientation, for the convolution algebra of $\mathscr{W}$ come from the exterior space $\Omega(\mathscr{W})$ of the Lie groupoid. How they are defined from the Lie groupoid $\mathscr{W}$ characterising the partial $^*$-algebras is explained in \cite{OkekeEgwe20202} following the construction of a smooth groupoid associated to a locally convex partial $^*$-algebra.

Given that $\mathscr{W}(f,-) \times [0,1) \cup \mathscr{W}(-,f) \times [0,1)$ is a neighbourhood of $f$ and a maximal open subset of the locally convex space $W^{k,p}(\Omega)$ allowing the definition of weak derivative. The intersection $\mathscr{W}(f,-) \cap \mathscr{W}(-,f) = \mathscr{W}(f,f)$ is the isotropy Lie group of $\mathscr{W}$ with left and right actions on the open submanifolds of smooth arrows $\mathscr{W}(f_\varepsilon,-), \mathscr{W}(-,f_\varepsilon) \equiv \mathscr{W}(f,-) \times [0,1), \mathscr{W}(-,f) \times [0,1)$. We have the containment of the isotropy Lie group $\mathscr{W}(f,f) \subset \mathscr{K}(\Omega)$. It follows that the infinitesimal neighbourhoods of identity arrow $(f,\delta,\varepsilon) \in \mathscr{W}(f,-) \times [0,1)$ generate the connected component $\mathscr{W}(f,-)$ (submanifold) of $\mathscr{W}$ on the Lie group $\mathscr{W}(f,f)$. The generation here is by various means of integrating an infinitesimal quantity to give a finite quantity, which is the sense the Sobolev  spaces are defined. This leads to the statement of the well known equivalence result on derivatives in this setting.
\begin{prop}
The pointwise or classical derivative is equivalent to the weak derivative pointwise almost everywhere on the Lie groupoid $\mathscr{W} \rightrightarrows W^{k,p}(\Omega)$.
\end{prop}
\begin{proof}
The infinite dimensional open submanifolds $\mathscr{W}(f,-)$ are supports of the left system of Haar measures $\mu^f$; they are approximated using the definition of local bisections $B_\ell(\mathscr{W})$ of the Lie groupoid $\mathscr{W}$ which are identified with the embedded submanifolds of $\mathscr{W}(f,-)$. Given the cohomogeneity-one structure of the space $\mathscr{W}$, the partial action of the Lie groupoid $\mathscr{W} \times \mathscr{W}(f,-) \to \mathscr{W}(f,-)$, which follows the global structure of its exterior space, is seen to be equivalent to the action of the Lie group $\mathscr{W}(f,f) \times \mathscr{W}(f,-) \to \mathscr{W}(f,-)$. This is given above as the invariance of the Sobolev space under the action of the smooth algebra $\mathscr{K}(\Omega) \times W^{k,p}(\Omega) \to W^{k,p}(\Omega)$, which expresses the equivalence of classical derivative and weak derivative by the convergence of the nets: $$(f_\varepsilon \to f, g_\varepsilon \to f,  \frac{f_\varepsilon - g_\varepsilon}{\varepsilon} \to X, \varepsilon \to 0) = (f,X,0),$$ for net of arrows in $\mathscr{W}(f,-) \times [0,1)$, where $X(f) = (f,X,0)$ expresses the classical derivative according to Connes \cite{Connes94}.
\end{proof}
\begin{coro}
The equivalence of the two derivatives results in projectability of the generalized vector fields $X_\varepsilon \to X$ on each $t$-fibre (resp. $s$-fibre) of $\mathscr{W}$.
\end{coro}
\begin{proof}
Given that each $t$-fibre $\mathscr{W}(f,-)$ (resp. $s$-fibre $\mathscr{W}(-,f)$) is a $\mathscr{K}(\Omega)_f$-slice as proved in \cite{OkekeEgwe2023}, which also follows on the action of the Lie groupoid $\mathscr{W}$ on the $t$-fibre above. It follows that the weak derivative $D_k$ corresponds to $W^{k,p}(\Omega)$ as a local cross section on the $t$-fibres $\mathscr{W}(f,-)$.

As shown in \cite{OkekeEgwe2023}, the Sobolev spaces $W^{k,p}(\Omega)$ are closed and embedded subspaces of distributions $\mathscr{D}'(\Omega)$, they are therefore subalgebras of the algebra of distributions. So, a distributional vector field $X \in \mathscr{D}(\R^n, T\R^n)$ is a limit of a net of smooth vector fields $(X_\varepsilon)_\varepsilon$ having the complete flows $\varphi_\varepsilon(t,.)$ converging to $\varphi(t,.)$ as the limit, such that $\varphi(t,.)$ is the flow of $X$. This defines projectability of vector fields through association of generalized vector fields to distributional vector fields, which was shown to be equivalent concepts in \cite{OkekeEgwe2023}. Thus, the pointwise equivalence of weak and classical derivatives in $W^{k,p}(\Omega)$ implies the projectability $X_\varepsilon \to X$ on each $t$-fibre of $\mathscr{W}$.
\end{proof}
\begin{rem}
The idea of association is used here because as noted in \cite{Grosseretal2001}, distributions are embedded into generalized algebra using association and not canonical injection, by making use of a generalized function (or quantity) associated with objects of interest. The choice of the generalized quantity is not unique.
\end{rem}
The above relates to the $1$-densities and the Haar system of measures on $\mathscr{W}$, for the projection of a generalized vector field $u \in \mathcal{G}(\R^n)$ is its distributional shadow $\omega \in \mathscr{D}(\R^n)$, satisfying the equality $\displaystyle \underset{\varepsilon \to 0}\lim \int_\Omega u_\varepsilon \nu = \langle \omega,\nu \rangle$, where $\nu$ is a compactly supported $1$-density on $\Omega \subset \R^n$.

The densities defined from these constitute the system of Haar measures for the Lie groupoid $\mathscr{W} \rightrightarrows W^{k,p}(\Omega)$. The smooth system of Haar measures satisfies the conditions given by Paterson (1999). Hence, a smooth chart as the one defined in \cite{OkekeEgwe20202} can be employed using $t$-fibrewise product as in \cite{Paterson99}. The smooth chart follows the partial multiplication on $L^p(\Omega)$, and its transition functions satisfy the cocycle conditions. Thus, as in \cite{OkekeEgwe20202}, a smooth left Haar system on the Lie groupoid $\mathscr{W} \rightrightarrows W^{k,p}(\Omega)$ is defined as follows.
\begin{defx}
A smooth left Haar system for the Lie groupoid $\mathscr{W} \rightrightarrows W^{k,p}(\Omega)$ is a family $\{\mu^f\}_{f \in W^{k,p}(\Omega)}$ where each $\mu^f$ is a positive, regular Borel measure on the open submanifold $\mathscr{W}(f,-)$ such that: \\ (i) If $(B_f,\psi)$ is a $t$-fibrewise product open subset of $\mathscr{W}$, $B_f \simeq t(B_f) \times W$, for $W \subset \R^k$ and if $\mu_W = \mu|_W$ is Lebesgue measure on $\R^k$, then for each $f \in t(B_f)$, the measure $\mu^f \circ \psi_f$ is equivalent to $\mu_W$, since $\psi_f : B_f\cap \mathscr{W}(f,-) \to \R^k$ is a diffeomorphism and their R-N derivative is the function $\kappa(f,w) = d(\mu^f \circ \psi_f)/d\mu_W(w)$ belonging to $C^\infty(t(B_f) \times W)$ and is strictly positive.\\
(ii) When $\kappa$ is restricted to $f \in W^{k,p}(\Omega)$, we have $\displaystyle \kappa_o(f) = \int_{\mathscr{W}(f,-)}\kappa d\mu^f$, which belongs to $C_c(W^{k,p}(\Omega))$.\\
(iii) For any $\gamma_1 \in \mathscr{W}$ and $F \in C^\infty_c(\mathscr{W})$, we have \[ \int_{\mathscr{W}(s(\gamma_1),-)} F(\gamma_1\gamma_2)d\mu^{s(\gamma_1)}(\gamma_2) = \int_{\mathscr{W}(t(\gamma_1),-)}F(\gamma_3)d\mu^{t(\gamma_1)}(\gamma_3), \] where $\gamma_1,\gamma_2, \gamma_3 \in \mathscr{W}$ and $\gamma_1\gamma_2 = \gamma_3$.
\end{defx}
Given the normalized $\frac{1}{2}$-densities $|\Omega|^{1/2}_\gamma$ forming the fibre over an arrow $\gamma \in \mathscr{W}$, with $s(\gamma) = g, t(\gamma) = f$; a density $\phi \in \Omega(\mathscr{W})_\gamma \subset C^\infty_c(\mathscr{W},\Omega(\mathscr{W}))$ is a map on the Lie groupoid $\mathscr{W}$ taking values in the exterior space $\Omega(\mathscr{W})$ of the groupoid. It defines a functional $\wedge^kT_\gamma(\mathscr{W}(-,f))\otimes \wedge^kT_\gamma(\mathscr{W}(g,-)) \to \mathbb{C}$ (or $\R$). The convolution algebra is defined over 1-densities on the external space $\Omega(\mathscr{W})$ of the Lie groupoid $\mathscr{W}$ based on the multiplication (direct product) of the densities following the composition rules of the groupoid arrows. Hence, a convolution algebra of the groupoid $\mathscr{W} \rightrightarrows W^{k,p}(\Omega)$ is defined on the space of sections $C^\infty_c(\mathscr{W},\Omega(\mathscr{W})) \subset L^2(\mathscr{W})$ of the line bundle, with the convolution product $\phi * \varphi$ for $\phi,\varphi \in C^\infty_c(\mathscr{W},\Omega(\mathscr{W}))$ given as \[ \phi * \varphi(\gamma) = \int_{\gamma_1\circ \gamma_2 = \gamma}\phi(\gamma_1)\varphi(\gamma_2) = \int_{\mathscr{W}(t(\gamma),-)}\phi(\gamma_1)\varphi(\gamma_1^{-1}\gamma). \]
This is an integral of sections on the manifold $\mathscr{W}(t(\gamma),-)$ since $\phi(\gamma_1)\varphi(\gamma_1^{-1}\gamma)$ is a 1-density.

So, the convolution algebra of $C^\infty(\mathscr{W},\Omega(\mathscr{W}))$ encodes transformations or partial symmetries of the partial $^*$-algebra on base space in a smooth and integrable way. The representation of the Lie groupoid is given in the sequel.

\section{Representation of the Lie Groupoid}
The representation of the Lie groupoid $\mathscr{W} \rightrightarrows W^{k,p}(\Omega)$ follows the definitions in \cite{OkekeEgwe20202}. The background is from \cite{Renault80}, \cite{Hahn78}, and \cite{Paterson99}. For the purpose of showcasing the simplification achieved through the Lie groupoid characterisation, there will be two representations. The first will be the representation of the Lie groupoid on the Hilbert bundle $\mathcal{H}$, whereby the arrows will define unitary elements of the bundle space. The second is the representation of the Lie groupoid algebra $C(\mathscr{W})$, the sections defining the groupoid algebra, on the sections of the Hilbert bundle $\mathcal{H}$. The identification of the two sections will lead to the definition of the left regular representation of the Lie groupoid $\mathscr{W}$ on the Hilbert bundle. The regular representations is usually employed to define the $C^*$-algebra representation of the Lie groupoid. This will be given in a different paper with application.

\subsection{The Unitary Representation of the Lie Groupoid}

The simple containment $C^\infty_c(\mathscr{W},\Omega(\mathscr{W})) \subset L^2(\mathscr{W})$ of the line bundles motivates the representation of the Lie groupoid and its algebra on the Hilbert bundle over the base space. Given the Lie groupoid $\mathscr{W} \rightrightarrows W^{k,p}(\Omega)$, each identity object $(\delta,f) \in W^{k,p}(\Omega)$ corresponds to a Haar measure $\mu^f$ supported on $\mathscr{W}(f,-)$ (or $\mu_f$ supported on $\mathscr{W}(-,f)$). A representation $\ell$ of the locally convex Lie groupoid $\mathscr{W}$ is defined on the Hilbert bundle $\mathcal{H} = \{H_f = L^2(\mathscr{W}(f,-),\mu^f): f \in W^{k,p}(\Omega)\}$.

Just as in \cite{OkekeEgwe20202}, the representation of $\mathscr{W} \rightrightarrows W^{k,p}(\Omega)$ is given as a triple $(W^{k,p}(\Omega),\mathcal{H},\nu)$, where $W^{k,p}(\Omega)$ a locally convex Hausdorff space, and $\nu$ is an invariant or quasi-invariant probability measure on $W^{k,p}(\Omega)$, and $\mathcal{H} = \{H_f = L^2(\mathscr{W}(f,-),\mu^f): f \in W^{k,p}(\Omega)\}$ is the collection of Hilbert spaces indexed by the elements of $W^{k,p}(\Omega)$.

Given that the basic measure $\nu$ is quasi-invariant, and the $\nu$-corresponding measures $m, m^2$ on $\mathscr{W}, \mathscr{W}^{(2)}$ satisfy the requirements for measurability of inversion and product. The representation $\ell$ of the Lie groupoid $\mathscr{W} \rightrightarrows W^{k,p}(\Omega)$ on the Hilbert space $\mathcal{H} = \{H_f = L^2(\mathscr{W}(f,-),\mu^f): f \in W^{k,p}(\Omega)\}$--the $L^2$-space of the Hilbert bundle over $W^{k,p}(\Omega)$--is done in such way that each arrow $\gamma = (f,g)$ defines a unitary element $\ell(\gamma) : H_{s(\gamma)} \to H_{t(\gamma)}$, in keeping with multiplication of arrows on $\mathscr{W}$; such that $\gamma_1\gamma_2 \; \implies \; \ell(\gamma_1)\circ \ell(\gamma_2) = \ell(\gamma_1\gamma_2)$ a.e. The representation is defined as follows.
\begin{defx}
A representation of the Lie groupoid $\mathscr{W} \rightrightarrows W^{k,p}(\Omega)$ is defined by the Hilbert bundle $(W^{k,p}(\Omega),\{H_f\},\nu)$ where $\nu$ is a quasi-invariant measure on $W^{k,p}(\Omega)$ (defining the associated measures $m, m^{-1}, m^2, m_o$), and each arrow $\gamma \in \mathscr{W}$ defines a unitary element $$\ell(\gamma) : H_{s(\gamma)} \to H_{t(\gamma),}$$ satisfying the following properties. \\ (i) $\ell(f)$ is the identity map on $H_f$ for all $f \in W^{k,p}(\Omega)$; \\ (ii) $\ell(\gamma_1\gamma_2) = \ell(\gamma_1)\ell(\gamma_2)$ for $m^2$-a.e. $(\gamma_1,\gamma_2) \in \mathscr{W}^{(2)}$; \\ (iii) $\ell(\gamma)^{-1} = \ell(\gamma^{-1})$ for $m$-a.e. $\gamma \in \mathscr{W}$; \\ (iv) for any pair of maps $\varphi, \psi \in L^2(W^{k,p}(\Omega),\{H_f\},\nu)$, the function \[\gamma \mapsto \langle \ell(\gamma)\varphi(s(\gamma)), \psi(t(\gamma))\rangle  \] is $m$-measurable on $\mathscr{W}$.
\end{defx}
The notations $(W^{k,p}(\Omega),\mathcal{H},\nu)$, $(\nu,\mathcal{H},\ell)$, or simply $\ell$ when $\nu$ and $\mathcal{H} = \{H_f : f \in W^{k,p}(\Omega)\}$ is understood, are used for the representation. The inner product is well defined since $\varphi(s(\gamma)) \in H_{s(\gamma)}$ and translated by $\ell(\gamma)$ to $\ell(\gamma)\varphi(s(\gamma)) \in H_{t(\gamma)}$, and $\psi(t(\gamma)) \in H_{t(\gamma)}$. So the inner product defined is on the fibre $H_{t(\gamma)}$.

The unitary representation of the Lie groupoid $\mathscr{W}$ is used to define the regular representation of the Lie groupoid algebra $C_c(\mathscr{W})$ as an integral representation defined by the integration of $\gamma \mapsto \varphi(\gamma)\ell(\gamma)$ over $\mathscr{W}(f,-)$, for $f \in W^{k,p}(\Omega), \varphi \in C_c(\mathscr{W})$, with respect to $\mu^f$, which is a pooling together for all $f \in W^{k,p}(\Omega)$, resulting in a bundle of Hilbert spaces $\{H_f\}_{f \in W^{k,p}(\Omega)}$ with a probability measure $\nu$ on the space of units $W^{k,p}(\Omega)$.

The pooling together requires a modular function because it is achieved through integrating with respect to the probability measure $\nu$ on $W^{k,p}(\Omega)$ possessing the \emph{quasi-invariant} property. the modular function is needed for the representation as defined in \cite{OkekeEgwe20202}. The associated measures $m, m^{-1}, m^2, m_o$ to $\nu$ and $\mu^f$ are also defined there for the representations of the Lie groupoid algebra $C_c(\mathscr{W})$.

\subsection{The Representation of the Lie Groupoid Algebra $C_c(\mathscr{W})$}
Since the local bisections of the Lie groupoid has been shown to encode the local or partial symmetries defined by the weak differential operators giving the local dynamics of the action of the smooth algebra on the base space, the definition of sections on the Hilbert bundle are shown to correspond to the sections defined by $1$-densities on the Lie groupoid. The section of the bundle and its net which determine the inner product norm are defined following \cite{Paterson99}.
\begin{defx}
A function $\varphi : W^{k,p}(\Omega) \to \underset{f \in W^{k,p}(\Omega)}\bigsqcup L^2(\mathscr{W}(f,-),\mu^f)$ defined by $f \mapsto \varphi(f) \in L^2(\mathscr{W}(f,-),\mu^f)$ is called a \emph{section} of the bundle.
\end{defx}
The image of the sections are functions defined on the arrows terminating at $f \in W^{k,p}(\Omega)$, which are integrable. This makes them element of $C^\infty(\mathscr{W},\Omega(\mathscr{W})) \simeq C^\infty(U)$. The latter are diffeomorphisms on an open subspace $U \subset W^{k,p}(\Omega)$, isomorphic to the subspaces corresponding to local bisections $B_\ell(\mathscr{W})$. A net of smooth sections which has a finer structure and is more suited to the distribution space replaces a sequence of sections as used in \cite{Paterson99}. We define a fundamental net as in \cite{OkekeEgwe20202}, following \cite{Paterson99}, and proof a result from \cite{OkekeEgwe20202} in this context as follows.
\begin{defx}
A net $(\varphi_\tau)$ of sections is \emph{fundamental} whenever for each pair of indices $\eta,\tau$ the function $\displaystyle f \mapsto \langle \varphi_\eta(f),\varphi_\tau(f) \rangle = \int \varphi_\eta(t\circ \varphi_\tau(f))\varphi_\tau(f)d\nu(f)$ is $\nu$-measurable on $W^{k,p}(\Omega)$; and for each $f \in W^{k,p}(\Omega)$, the images $\varphi_\tau(f)$ of the net span a dense subspace of $H_f = L^2(\mathscr{W}(f,-),\mu^f)$. 
\end{defx}
\begin{thm}
A net $(\sigma_\tau)$ of local bisection $B_\ell(U)$ of $\mathscr{W} \rightrightarrows W^{k,p}(\Omega)$ corresponds to a fundamental net of the sections of the Hilbert space $L^2(\mathscr{W}(f,-),\mu^f)$.
\end{thm}
\begin{proof}
Given that $U \subset W^{k,p}(\Omega)$ is a trivialization, then $H_f = L^2(\mathscr{W}(f,-),\mu^f)$ is an open submanifold corresponding to the image of a local bisection $B_\ell(U)$-a Lie pseudogroup. Then the $\nu$-measurability of the net of sections $\sigma_\tau$ follows from the $\nu$-measurability of $f$ which is the inverse image of the arrows under the target map $t : \mathscr{W}(-,f) \to W^{k,p}(\Omega)$ which is inverse of the section on the locally convex space $W^{k,p}(\Omega)$. Thus, the net of local bisections $\sigma_\tau$ corresponds to a net of sections $\varphi_\tau$ on the Hilbert bundle, and the image of the two $\sigma(f), \varphi(f)$ are equally integrable functions.

The resulting Hilbert bundle $\mathcal{H} = L^2(W^{k,p}(\Omega),\{H_f\},\nu)$ is the space of measurable sections $\varphi$ (equivalence classes) defining a $\nu$-integrable function $f \mapsto ||\varphi(f)||^2_2$, with inner product $\displaystyle \langle \varphi,\psi \rangle = \int_{W^{k,p}(\Omega)}\langle \varphi(f),\psi(f) \rangle d\nu(f)$, which defines the norm $\displaystyle ||\varphi||^2 = \int_{W^{k,p}(\Omega)} \langle \varphi(f),\varphi(f) \rangle d\nu(f)$. Thus, the $L^2$-norm is $\displaystyle ||\varphi||^2_2 = \int_{W^{k,p}(\Omega)}||\varphi(f)||^2d\nu(f)$.
\end{proof}

\begin{lem}
The equivalence class $[\varphi]$ of sections is determined by the smooth action of the algebra $\mathscr{K}(\Omega)$ by its smooth net.
\end{lem}
\begin{proof}
This follows from the net action on the fibre translation $(\phi_\tau)(\gamma)\gamma^{-1}$ defining the net of bisections $(\varphi_\tau)$ in $[\varphi]$; or from our proof of the local bisections $B_\ell(U)$ which is a Lie pseudogroup to be invariant under $\mathscr{K}(\Omega)$-action.
Thus, the equivalence class $[\varphi]$ of a section $\varphi$ has the inner product action $\langle \varphi(f),\varphi_\tau(f) \rangle$ of a fundamental net $(\varphi_\tau)$ in terms of measurability, in keeping with Theorem 2.1.

The fundamental net $(\varphi_\tau)$ spans a dense subspace of each $L^2(\mathscr{W}(f,-),\mu^f)$ on any $f \in L^p(\Omega)$, and defines the equivalence classes of sections of the Hilbert bundle $\mathcal{H} = L^2(W^{k,p}(\Omega),\{H_f\},\nu)$.
\end{proof}

\begin{rem}
The dimension of the Hilbert bundle is constant only when the Hilbert space $L^2(\mathscr{W}(f,-),\mu^f)$ is the same for all $f \in W^{k,p}(\Omega)$. Cf. \cite{Paterson99}. The locally convex Lie groupoid $\mathscr{W} \rightrightarrows W^{k,p}(\Omega)$ has fibres which are not all of same dimension since the $t$-fibres $\mathscr{W}(f,-)$ are not the same for every $f \in W^{k,p}(\Omega)$. The fundamental nets result from the $\mathscr{K}(\Omega)$-action on universal multipliers.
\end{rem}

Subsequently, we represent the sections defining the Lie groupoid algebra on the sections of the Hilbert bundle. So, let $C_c(\mathscr{W})$ be space of continuous (smooth) sections of $\mathscr{W}$, we define a representation of the algebra on the bundle $\mathcal{H}$ by identification of each $\varphi \in C_c(\mathscr{W})$ with the section $f \to \varphi|_{\mathscr{W}(f,-)} \in C_c(\mathscr{W}(f,-)) \subset L^2(\mathscr{W}(f,-))$.
\begin{defx}
Given any pair of sections $\varphi,\psi \in C_c(\mathscr{W})$, define a $\nu$-measurable map by inner product $\gamma \mapsto \langle \varphi(\gamma), \psi(\gamma) \rangle, C_c(\mathscr{W}) \to L^2(\mathcal{H})$. This is well defined since $\varphi\overline{\psi} \in B(\mathcal{H})$ and the restriction $(\varphi\overline{\psi})|_{W^{k,p}(\Omega)} = (\varphi\overline{\psi})^o \in B_c(W^{k,p}(\Omega))$.
\end{defx}
The generation of the Hilbert bundle from these sections is shown in \cite{OkekeEgwe20202}. Cf. \cite{Paterson99}, Definition 2.2.1. Given this identification of $C_c(\mathscr{W})$ with the sections of $L^2(\mathcal{H})$, where $\mathcal{H} = \{H_f : f \in W^{k,p}(\Omega)\}$ is a Hilbert bundle, the inner product $\langle \varphi_1(f),\varphi_2(f) \rangle$ is interpreted to be the product of two sections defined above, in keeping with multiplication of arrows, as $\displaystyle \int (\varphi_1 \star \varphi_2)(f)d\nu(f) = \int \varphi_1(t \circ \varphi_2(f))\varphi_2(f)d\nu(f)$. The main result on the partial dynamical systems of the locally convex partial $^*$-algebra $\Gamma$ now follows.

\begin{thm}
The convolution product with a fundamental net defines the partial dynamical system of the locally convex partial $^*$-algebra $\Gamma$ given by the regular action of Lie groupoid $\mathscr{W}$ on the $t$-fibres.
\end{thm}
\begin{proof}
The convolution product $\phi * \varphi$ for $\phi,\varphi \in C^\infty_c(\mathscr{W},\Omega(\mathscr{W}))$ given as \[ \phi * \varphi(\gamma) = \int_{\gamma_1\circ \gamma_2 = \gamma}\phi(\gamma_1)\varphi(\gamma_2) = \int_{\mathscr{W}(t(\gamma),-)}\phi(\gamma_1)\varphi(\gamma_1^{-1}\gamma), \]
which is an integral of sections on the manifold $\mathscr{W}(t(\gamma),-)$ since $\phi(\gamma_1)\varphi(\gamma_1^{-1}\gamma)$ is a 1-density, as given above.
Given a fundamental net $\phi_\tau$ on $\mathscr{W}(t(\gamma),-)$ generating a dense subspace, the convolution product is given as
\[ \phi_\tau * \varphi(\gamma) = \int_{\gamma_1\circ \gamma_2 = \gamma}\phi_\tau(\gamma_1)\varphi(\gamma_2) = \int_{\mathscr{W}(t(\gamma),-)}\phi_\tau(\gamma_1)\varphi(\gamma_1^{-1}\gamma), \]
which is net of integral of sections representing the partial dynamical system of $\Gamma$. Since the convergence is in the Sobolev space $W^{k,p}(\Omega) \subset L^p(\Omega)$ embedded in the distribution space $\mathscr{D}'(\Omega)$ by the smooth action, it is a net of weak quantities converging to the classical derivatives and flows of the locally convex partial $^*$-algebra $\Gamma \subset L^p(\Omega) \times L^p(\Omega)$.
\end{proof}
The regular representation of the Lie groupoid $\mathscr{W}$ shows its regular action on the fibres and the Hilbert bundle.

\subsection{The Left Regular Representation of $\mathscr{W}$}
The natural representation of the Lie groupoid $\mathscr{W}$ is the left regular $\ell_r$ defined on the fibres by $H_f = L^2(\mathscr{W}(f,-))$ (see \cite{Paterson99}, p.107; \cite{Renault80}, p.55). This left regular representation is now given as a result and shown to be unitary.
\begin{prop}
The map $\ell_r(\gamma) : H_{s(\gamma)} \to H_{t(\gamma)}$ defined by $(\ell_r(\gamma))(\varphi)(\gamma_1) = \varphi(\gamma^{-1}\gamma_1)$, with $\varphi \in L^2(\mathscr{W}(s(\gamma),-))$ and $\gamma_1 \in \mathscr{W}(t(\gamma),-)$, is a unitary representation.
\end{prop}
\begin{proof}
First, $\ell_r(\gamma)$ is an extension of a bijective isometry \[L^1(\mathscr{W}(s(\gamma),-),\mu^{s(\gamma)}) \to L^1(\mathscr{W}(t(\gamma),-), \mu^{t(\gamma)}), \varphi \mapsto \gamma * \varphi; \] in the sense that $L^2(\mathscr{W}(s(\gamma),-),\mu^{s(\gamma)}) \subset L^1(\mathscr{W}(s(\gamma),-),\mu^{s(\gamma)})$; and it defines a translative (transitive) action of $\mathscr{W}$ on $t$-fibres.

Second, the restriction of the $\ell_r$ to $L^2(\mathscr{W}(s(\gamma),-),\mu^{s(\gamma)})$ is a representation of $\gamma \in \mathscr{W}$ as a unitary operator $\gamma * \varphi$ given as \[ L^2(\mathscr{W}(s(\gamma),-), \mu^{s(\gamma)}) \to L^2(\mathscr{W}(t(\gamma),-), \mu^{t(\gamma)}), (\gamma * \varphi)(\gamma_1) = \varphi(\gamma^{-1}\gamma_1) = \varphi(\gamma^{-1})\varphi(\gamma_1).\]  The restriction holds for $1 < p \leq \infty$ (see \cite{Paterson99}, p.34). So, $\ell_r$ is a (unitary) representation of $\mathscr{W} \rightrightarrows W^{k,p}(\Omega)$ on the $L^2$-space of the bundle $(L^p(\Omega),\mathcal{H},\nu)$ for it satisfies the other conditions of definition.
\end{proof}

\section{Conclusion}
The aim of this paper was to study the locally convex partial $^*$-algebra $(L^p(\Omega), \Gamma, \cdot,^*)$, with $\Gamma \subset L^p(\Omega) \times L^p(\Omega)$, using the Lie groupoid framework. This was achieved by embedding the partial algebra $(L^p(\Omega), \Gamma, \cdot,^*)$ into the Sobolev space $W^{k,p}(\Omega) \subset L^p(\Omega)$ closed subspace of $L^P(\Omega)$ under the weak derivative. The algebra was successively characterised as the Lie groupoid $\mathscr{W} \rightrightarrows W^{k,p}(\Omega)$ with the space of arrows defined by the relation $\Gamma$ in the Sobolev space $W^{k,p}(\Omega)$. The partial dynamical systems of the locally convex partial $^*$-algebra $(L^p(\Omega), \Gamma, \cdot,^*)$ is found to be the stable flows of the Hilbert bundle $L^2(W^{k,p}(\Omega), \mathcal{H}, \nu)$ under the convergence of the fundamental smooth nets of sections $\varphi_\tau$  of the Hilbert bundle.

Given that the Sobolev space $W^{k,p}(\Omega)$ is dense in $L^p(\Omega)$, the fundamental nets of sections of the Hilbert space $H_f = L^2(\mathscr{W}(f,-),\nu^f)$ having the regular action of the Lie groupoid $\mathscr{W}$ generate dense subspace of the Hilbert space $H_f$ which is also dense in the Sobolev space and in $L^p(\Omega)$.
The translative action of $\mathscr{W}$ on the $t$-fibres which are open subspaces of the locally convex Lie groupoid $\mathscr{W}$ can also be given by the action of the local bisections of the Lie groupoid, which are always diffeomorphic to the open subspaces of a Lie groupoid; that is, $L \longleftrightarrow \sigma$, where $L$ is an open subspace of $\mathscr{W}$ and $\sigma$ is a local bisection of $\mathscr{W}$. The action of the fundamental nets is therefore shown to represent the partial dynamical systems of the locally convex partial $^*$-algebra $(L^p(\Omega), \Gamma, \cdot,^*)$.

The use of the Lie groupoid characterisation $\mathscr{W} \rightrightarrows W^{k,p}(\Omega)$ easily demonstrated the partial dynamical systems of the locally convex partial $^*$-algebra $(L^p(\Omega), \Gamma, \cdot,^*)$ to be dense or contained on the $t$-fibres of the Lie groupoid. The Hilbert bundle was used to represent the partial algebra and its dynamics. To understand the locally convex partial $^*$-algebra $(L^p(\Omega), \Gamma, \cdot,^*)$ and its partial dynamical systems, we just need to concern ourselves with the fundamental nets of the sections of the Hilbert bundle $L^2(\mathcal{H}$.

\bibliographystyle{amsplain}

\providecommand{\bysame}{\leavevmode\hbox to3em{\hrulefill}\thinspace}

\end{document}